\documentclass[
final
]{dmtcs-episciences}

\usepackage[utf8]{inputenc}
\usepackage{amsmath,amsfonts,amsthm,amssymb}

\newtheorem{theorem}{Theorem}

\newtheorem{proposition}[theorem]{Proposition}

\usepackage[round]{natbib}

\title[On polynomials associated to Voronoi diagrams]{On polynomials associated to Voronoi diagrams of point sets and crossing numbers}

\author[Merc\`e Claverol et al.]
{
    Merc\`e Claverol\affiliationmark{1}\thanks{merce.claverol@upc.edu}
    \and Andrea de las Heras-Parrilla\affiliationmark{1}\thanks{andrea.de.las.heras@estudiantat.upc.edu}\\
    \and David Flores-Pe\~{n}aloza\affiliationmark{2}\thanks{dflorespenaloza@ciencias.unam.mx}
    \and Clemens Huemer\affiliationmark{1}\thanks{clemens.huemer@upc.edu}
    \and David Orden\affiliationmark{3}\thanks{david.orden@uah.es}
}
  
\affiliation{
  Universitat Polit\`ecnica de Catalunya, Barcelona, Spain\\
  Facultad de Ciencias, Universidad Nacional Aut\'onoma de M\'exico,  Ciudad de M\'exico, Mexico\\
  Universidad de Alcal\'a, Alcal\'a de Henares, Spain}
  

\begin{document}


\publicationdata{vol. 26:2}{2024}{14}{10.46298/dmtcs.12443}{2023-10-19; 2023-10-19; 2024-05-08}{2024-05-20}

\maketitle
\begin{abstract}    
    Three polynomials are defined for given sets $S$ of $n$ points in general position in the plane: The Voronoi polynomial with coefficients the numbers of vertices of the order-$k$ Voronoi diagrams of~$S$, the circle polynomial with coefficients the numbers of circles through three points of $S$ enclosing $k$ points of $S$, and the $E_{\leq k}$ polynomial with coefficients the numbers of (at most $k$)-edges of~$S$. We present several formulas for the rectilinear crossing number of $S$ in terms of these polynomials and their roots.
    We also prove that the roots of the Voronoi polynomial lie on the unit circle if, and only if, $S$ is in convex position. Further, we present bounds on the location of the roots of these polynomials.
\end{abstract}

\section{Introduction}

Let $S$ be a set of $n \geq 4$ points in general position in the plane, meaning that no three points of $S$ are collinear and no four points of $S$ are cocircular. 
The Voronoi diagram of order $k$ of $S$, $V_k(S)$, is a subdivision of the plane into cells such that points in the same cell have the same $k$ nearest points of $S$.
Voronoi diagrams have found many applications in a wide range of disciplines, see e.g.~\cite{A91, OBS00}. 
We define the \emph{Voronoi polynomial} $p_V(z)=\sum_{k=1}^{n-1} v_k z^{k-1}$, where $v_k$ is the number of vertices of $V_k(S)$.  

Proximity information among the points of $S$ is also encoded by the circle polynomial of $S$, which we define as $p_C(z)=\sum_{k=0}^{n-3} c_k z^k$, where $c_k$ denotes the number of circles passing through three points of $S$ that enclose exactly $k$ other points of $S$. 

The numbers $v_k$ and $c_k$ are related via the well-known relation
\begin{equation}\label{rel:ckvk}
v_k = c_{k-1}+c_{k-2}
\end{equation}
where $c_{-1}=0$ and $c_{n-2}=0$, see e.g. Remark 2.5 in~\cite{L03}.
These two polynomials $p_V(z)$ and $p_C(z)$ are especially interesting due to their connection to the prominent rectilinear crossing number problem. 

The rectilinear crossing number of a point set $S$, $\overline{cr}(S)$, is the number of pairwise edge crossings of the complete graph $K_n$ when drawn with straight-line segments on $S$, i.e. the vertices of $K_n$ are the points of $S$. 
Equivalently, $\overline{cr}(S)$ is the number of convex quadrilaterals with vertices in $S$. 
We denote $\overline{cr}(S)$ as $\alpha {{n}\choose{4}}$, with $0 \leq \alpha \leq 1$. Note that for $S$ in convex position, $\alpha=1$.  
The rectilinear crossing number problem consists in, for each $n$, finding the minimum value of $\overline{cr}(S)$ among all sets $S$ of $n$ points, no three of them collinear. 
This minimum is commonly denoted as $\overline{cr}(K_n)$. The limit of ${\overline{cr}(K_n)}/{{n}\choose{4}}$, when $n$ tends towards infinity, is the so-called rectilinear crossing number constant $\alpha^*$. This problem is solved  for $n\leq 27$ and $n=30$, and the current best bound for the rectilinear crossing number constant is $\alpha^* > 0,37997$; for more information, see the survey of~\cite{AFS12} and the web page~\cite{Aurl}.  
A fruitful approach to the rectilinear crossing number problem is proving bounds on the numbers of $j$-edges and of $(\leq k)$-edges of $S$, see~\cite{AC12, AF05, AGOR07, BS06, LVW04}. An (oriented) $j$-edge of $S$ is a directed straight line $\ell$ passing through two points of $S$ such that the  open half-plane bounded by $\ell$ and on the right of $\ell$ contains exactly $j$ points of $S$.  The number of $j$-edges of $S$ is denoted by $e_j$, and  $E_{\leq k} = \sum_{j=0}^{k} e_j$ is the number of $(\leq k)$-edges. We then consider the $E_{\leq k}$ polynomial $p_E(z)=\sum_{k=0}^{n-3} E_{\leq k} z^k$, which also encodes information on higher order Voronoi diagrams, since the number of $j$-edges $e_j$ is the number of unbounded cells of the order-$(j+1)$ Voronoi diagram of $S$, see e.g. Proposition 30 in~\cite{CdH21}. Note that $p_E(z)$ has no term $E_{\leq n-2}$.

For an illustration of the defined polynomials for a particular point set, see Figure~\ref{fig:1509}.

\begin{figure}[h!]
	\centering
		\includegraphics[scale=0.28,page=1]{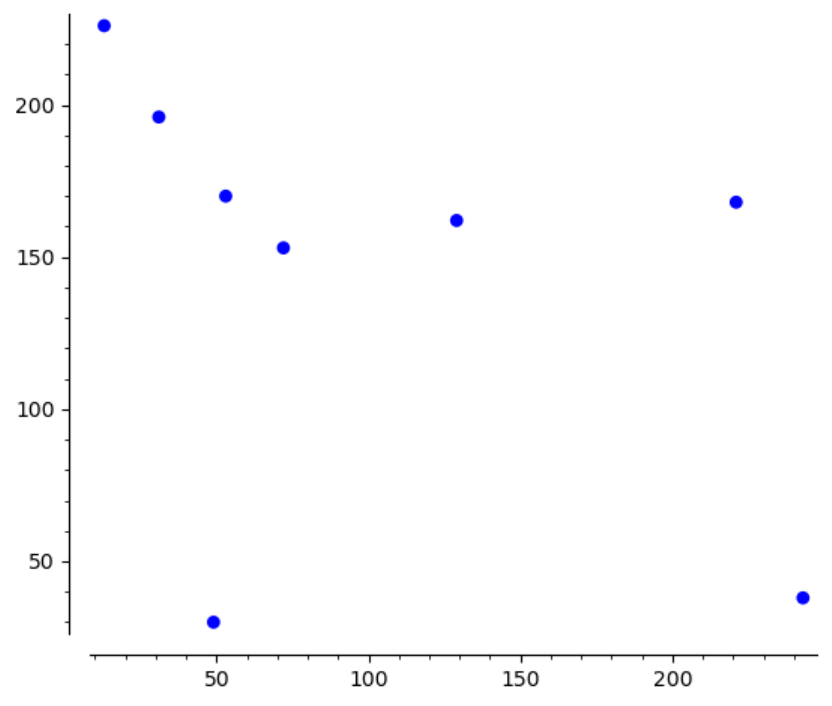}\quad
		\includegraphics[scale=0.28,page=2]{fig1509.pdf}\quad
	\includegraphics[scale=0.28,page=3]{fig1509.pdf}
	\caption{Left: 1591-th entry of the order type database for 8 points, from ~\cite{OrderTypesDBurl}.  With complex stream plots of its Voronoi polynomial (center): $p_V(z)=10 + 23z + 27z^2 + 24z^3 + 17z^4  + 9z^5 + 2z^6$, and its $E_{\leq k}$ polynomial (right): $p_E(z)=4 + 13z + 22z^2 + 34z^3 + 43z^4 + 52z^5$; roots are red points.
	}
	\label{fig:1509}
\end{figure}

For a point set $S$, we show that $\overline{cr}(S)$ appears in the first derivatives of these three polynomials when evaluated at $z=1$ and, in addition, we obtain appealing formulas for $\overline{cr}(S)$ in terms of the roots of the polynomials. Motivated by this, we study the location of such roots, showing several bounds on their modulus. As a particular result, we also prove that the roots of the Voronoi polynomial lie on the unit circle, $\{z:|z|=1\}$, if, and only if, $S$ is in convex position. 
Polynomials with zeros on the unit circle have been studied for instance in \cite{CH20, Chen95, LL04, LS13}.
  
Furthermore,  the circle polynomial  comes into play when considering the random variable $X$ that counts the number of points of $S$ enclosed by the circle defined by three points chosen uniformly at random from~$S$. The probability generating function of $X$ is $p_C(z)/{{n}\choose{3}}$. In~\cite{MS19} a central limit theorem 
for random variables with values in $\{0,\ldots,n\}$ was shown, under the condition that the variance is large enough and that no root of the probability generating function is too close to $1 \in \mathbb{C}.$ We show that the random variable $X$ does not approximate a normal distribution, and use the result from~\cite{MS19} to derive that $p_C(z)$ has a root close to $1 \in \mathbb{C}.$  

We first state in Section~\ref{sec:known} the known relations for Voronoi diagrams, circles enclosing points, and $j$-edges, that we will use. 
Then, in Section~\ref{sec:polynomials}, we apply them to obtain properties of the three polynomials.
Section \ref{sec:rootsOfPolynomials} is on the roots of the polynomials. 
Finally, Section~\ref{sec:discussion} discusses open problems, the related polynomial of $j$-edges, and conclusions.

Throughout this work, points $(a,b)$ in the plane are identified with complex numbers $z=a+ib$. To avoid cumbersome notation we omit indicating the point set $S$ where it is clear from context; for example, we write $p_C(z)$ instead of $p^S_C(z)$.

\section{Known relations}\label{sec:known}

In this section, necessary results for the current paper 
are presented. 
These consist in the relations between
 the number $c_k$ of circles enclosing $k$ points of $S$, the crossing number $\overline{cr}(S)$, the number  $E_{\leq k}$ of (at most $k$)-edges of $S$, and the number $v_k$ of vertices of the Voronoi diagram of order $k$ of $S$.

A main source is the work by \cite{L82}, from where several of the following formulas can be obtained.

\begin{itemize}
    \item
    For any point set $S$, and $0\leq k\leq n-3$, it holds that, see~\cite{A04, CS89, CdH21, L82,L03}, 
    \begin{equation}\label{eqpairs}
        c_k+c_{n-k-3}=2(k+1)(n-k-2).
    \end{equation}

    \item
    From~\cite{circle12} we get the following two equations.
    \begin{equation}\label{eq:main}
        \sum_{k=0}^{n-3} k \cdot c_k =
        {{n}\choose{4}}+\overline{cr}(S)= (1+\alpha){{n}\choose{4}}.
    \end{equation}
    This was essentially also obtained in~\cite{U04}, though not stated in terms of $\overline{cr}(S)$.
    
    \begin{equation}\label{eq:mainsquare}
        \sum_{k=0}^{n-3} k^2 \cdot c_k ={{n}\choose{5}}+{{n}\choose{4}}+(n-3)\overline{cr}(S).
    \end{equation}

    \item 
    For $k \leq \frac{n-3}{2}$ it holds that, see Lemma 3.1 in~\cite{C19},
        \begin{equation}\label{inequ1}
            c_k \geq (k+1)(n-k-2)
        \end{equation}
        and
        \begin{equation}\label{inequ2}
            c_{n-k-3} \leq (k+1)(n-k-2).
        \end{equation}
        
    \item 
    For every set $S$ of $n$ points in general position, the relation between $E_{\leq k}$  and $c_k$ is, see e.g. \\ Property~33 in~\cite{CdH21},
    \begin{equation}\label{circandE}
    	c_k + E_{\leq k}=(k+1)(2n-k-2).
    \end{equation}
    For a point set $S$ in convex position, we have equality in Equations~(\ref{inequ1}) and~(\ref{inequ2}) which can be derived from Equations~(\ref{circandE}) and $E_ {\leq k} =(k+1)n$ for $S$ in convex position,
    and therefore
    \begin{equation}\label{eqn:concave}
        2c_k = c_{k-1} + c_{k+1} + 2.		
    \end{equation}
    Then,  the number of vertices of $V_k(S)$, for $S$ in convex position fulfills, see e.g.~Property 34, Equation (4) in \cite{CdH21},
    \begin{equation}\label{eqn:numverticespv}
        v_k=c_{k-1}+c_{k-2}=(2k-1)n - 2k^2\ .
    \end{equation}
    
    Note that, for $S$ in convex position this implies that
    \begin{equation}\label{eq:palindromic}
        v_k=v_{n-k}.
    \end{equation}

\end{itemize}

\section{Properties of the Voronoi, circle and $E_{\leq_k}$ polynomials}\label{sec:polynomials}

In the following section, we will show the relations between the Voronoi and the circle polynomials, $p_V(z)$ and $p_C(z)$. 
Some additional properties for $p_V(z), p_C(z)$ and $p_E(z)$ will also be presented. 
Finally, we will give a family of formulas for the crossing number in terms of the coefficients of these three polynomials.

\begin{proposition}\label{prop:equalpoly}
    For every set $S$ of $n$ points in general position, the circle polynomial $p_C(z)=\sum_{k=0}^{n-3} c_k z^{k}$ and the Voronoi polynomial $p_V(z)=\sum_{k=1}^{n-1} v_k z^{k-1}$ satisfy: 
    \begin{equation}
        p_V(z) = (1+z) p_C(z).
    \end{equation}
\end{proposition}
\begin{proof}
    This follows from the property $v_k = c_{k-1}+c_{k-2},$ for $1\leq k \leq n-1$, where $c_{-1}=0$ and $c_{n-2}=0$.
    Then, \[p_V(z)=c_0+ \sum_{k=2}^{n-2} (c_{k-1}+c_{k-2})z^{k-1}+ c_{n-3}z^{n-2} = p_C(z) + z\cdot p_C(z).\]
\end{proof}

\begin{proposition}\label{prop:circleeq}
    For every set $S$ of $n$ points in general position,
    the circle polynomial $p_C(z)=\sum_{k=0}^{n-3} c_k z^{k}$ satisfies:
    \begin{enumerate}
        \item $p_C(1)={{n}\choose{3}}.$\label{pc1}
        \item $p'_C(1)= {{n}\choose{4}}+ \overline{cr}(S).$
        \item $p''_C(1)={{n}\choose{5}}+(n-4)\overline{cr}(S).$
        \item $p_C(-1)=\frac{n-1}{2}$ for $n$ odd. \label{pc-1}
    \end{enumerate}
\end{proposition}
\begin{proof}
  Since there are ${{n}\choose{3}}$ different circles passing through three different points from $S$, the first claim follows.
  The second one follows from Equation~(\ref{eq:main}).
  The third equality follows from Equations~(\ref{eq:main}) and~(\ref{eq:mainsquare}).
  Finally, we use Equation~(\ref{eqpairs}) to obtain the fourth equality; note that from Equation~(\ref{eqpairs}), for $n$ odd it follows that $c_{\frac{n-3}{2}}=\frac{(n-1)^2}{4}$. 
  \[p_C(-1)=\sum_{k=0}^{n-3}(-1)^k c_k=\sum_{k=0}^{\frac{n-3}{2}-1}\left((-1)^k 2(k+1)(n-2-k)\right)+(-1)^{\frac{n-3}{2}}c_{\frac{n-3}{2}}=\frac{n-1}{2}.\]
\end{proof}

When we consider the Voronoi polynomial 
we get:

\begin{proposition}
    For every set $S$ of $n$ points in general position, the Voronoi polynomial \\ $p_V(z)=\sum_{k=1}^{n-1} v_k z^{k-1}$ satisfies:
    \begin{enumerate}
        \item $p_V(1)=2{{n}\choose{3}}.$
        \item $p'_V(1)= {{n}\choose{3}}+2{{n}\choose{4}} + 2\overline{cr}(S).$
        \item $p''_V(1)=2{{n}\choose{4}}+2{{n}\choose{5}}+2(n-3)\overline{cr}(S).$
        \item $p_V(-1)=0.$
        \item $p'_V(-1)=\frac{n-1}{2}$ for $n$ odd.
    \end{enumerate}
\end{proposition}
\begin{proof}
    This follows from Proposition~\ref{prop:equalpoly} and Proposition~\ref{prop:circleeq}. From $p_V(z) = (1+z) p_C(z)$, we get
    \[p'_V(z)=p_C(z) + (1+z)\cdot p'_C(z)\] and
    \[p''_V(z)=2p'_C(z) + (1+z)\cdot p''_C(z).\]
    Then just substitute from Proposition~\ref{prop:circleeq}.
    The fourth property follows directly from Proposition~\ref{prop:equalpoly} and was already noted by~\cite{L03}. The last property follows from taking derivatives in Proposition~\ref{prop:equalpoly}.
\end{proof}

When we consider the $E_{\leq_k}$ polynomial 
we get:
\begin{proposition}\label{prop:edgepoly}
    For every set $S$ of $n$ points in general position, the $E_{\leq_k}$ polynomial $p_E(z)=\sum_{k=0}^{n-3} E_{\leq_k} z^{k}$ satisfies:
    \begin{enumerate}
        \item $p_E(1)=3{{n}\choose{3}}.$
        \item $p'_E(1)=
        9{{n}\choose{4}}- \overline{cr}(S).$
        \item $p''_E(1)=
        35{{n}\choose{5}}-(n-4)\overline{cr}(S).$
        \item $p_E(-1)=\frac{n(n-1)}{2}$ for $n$ odd.
    \end{enumerate}
\end{proposition}
\begin{proof}  
   The first equality follows from Equation~(\ref{circandE})
   by summing over all $k$, and Proposition~\ref{prop:circleeq}, part~\ref{pc1}. Note that 
   $\sum_{k=0}^{n-3} \left((k+1)(2n-k-2)\right)=4{{n}\choose{3}}$.
   \[p_E(1)=\sum_{k=0}^{n-3} E_{\leq k}=\sum_{k=0}^{n-3} \left((k+1)(2n-k-2)-c_k\right)=4{{n}\choose{3}}-p_C(1)=3{{n}\choose{3}}\] 
   
   The second and third equalities follow from Equations~(\ref{eq:main}) and~(\ref{circandE}).
   
   The fourth equality follows from Equation~(\ref{circandE}) and Proposition~\ref{prop:circleeq}, part~\ref{pc-1}. Note that for $n$ odd: 
   
   $\sum_{k=0}^{n-3}(-1)^k \left((k+1)(2n-k-2)\right)=\frac{n^2-1}{2}$.
   \[p_E(-1)=\sum_{k=0}^{n-3}(-1)^k E_{\leq k}=\sum_{k=0}^{n-3}(-1)^k \left((k+1)(2n-k-2)\right)-p_{C}(-1)=\frac{n(n-1)}{2}\] 
\end{proof}

From the following result of Aziz and Mohammad, see Lemma 1 in~\cite{A80}, we get an intriguing family of formulas for the rectilinear crossing number in terms of the coefficients of the three studied polynomials, and the roots of a non-zero complex number $a\neq -1$. 

\begin{theorem}[Aziz and Mohammad, 1980]
\label{thm:Aziz}
    If $P(z)$ is a polynomial of degree $n$ and $z_1,\ldots, z_n$ are the zeros of $z^n+a$, where $a\neq -1$ is any non-zero complex number, then for any complex number $t$,
    \[tP'(t)=\frac{n}{1+a}P(t)+\frac{1+a}{n a}\sum_{k=1}^{n} P(t z_k) \frac{z_k}{(z_k-1)^2}.\]
\end{theorem}

\begin{proposition}\label{prop:Aziz}
    The coefficients of the polynomials $p_V(z)$, $p_C(z)$ and $p_E(z)$ satisfy
    \begin{equation} \label{eq:cr-1}
        \emph{\bf 1) }
        \overline{cr}(S)= \sum_{j=0}^{n-3}c_j \left( \frac{4}{3(n-3)}\sum_{k=1}^{n-3} \frac{z_k^{j+1}}{(z_k-1)^2} \right), 
    \end{equation}
    where the $z_k$ are the $(n-3)$-th roots of~$-3$.
    \begin{equation}\label{eq:cr-2}
        \emph{\bf 2) }
        \overline{cr}(S)=\sum_{j=1}^{n-1}v_j \left( \frac{2}{3(n-1)}\sum_{k=1}^{n-1} \frac{z_k^{j}}{(z_k-1)^2} \right),
    \end{equation}
    where the $z_k$ are the $(n-1)$-th roots of~$-3$.

    \begin{equation}\label{eq:cr-3}
        \emph{\bf 3) }
        \overline{cr}(S)=\sum_{j=0}^{n-3}E_{\leq j} \left( \frac{-4}{n-3}\sum_{k=1}^{n-3} \frac{z_k^{j+1}}{(z_k-1)^2} \right),
    \end{equation}
    where the $z_k$ are now the $(n-3)$-th roots of $-\frac{1}{3}$.
\end{proposition}
\begin{proof}
    Using Theorem~\ref{thm:Aziz} for $t=1$ and for the circle polynomial $p_C(z)$, we get for any non-zero complex number $a\neq -1$,
    \begin{equation}
        {{n}\choose{4}}+\overline{cr}(S) = \frac{n-3}{1+a}{{n}\choose{3}}+ \frac{1+a}{(n-3) a}\sum_{k=1}^{n-3} \sum_{j=0}^{n-3}c_j  \frac{z_k^{j+1}}{(z_k-1)^2}
    \end{equation}
    
    \noindent Taking $a=3$ we get 
    \begin{equation}
        \overline{cr}(S)=\frac{4}{3(n-3)}\sum_{k=1}^{n-3} \sum_{j=0}^{n-3}c_j  \frac{z_k^{j+1}}{(z_k-1)^2}= \sum_{j=0}^{n-3}c_j \left( \frac{4}{3(n-3)}\sum_{k=1}^{n-3} \frac{z_k^{j+1}}{(z_k-1)^2} \right),
    \end{equation}
    where the $z_k$ are the $(n-3)$-th roots of $-3$.\\
    
    Using Theorem~\ref{thm:Aziz} for $t=1$ and for the Voronoi polynomial $p_V(z)$, we get for any non-zero complex number $a\neq -1$,
    \begin{equation}
        {{n}\choose{3}}+2{{n}\choose{4}}+2\overline{cr}(S) = \frac{n-1}{1+a}2{{n}\choose{3}}+ \frac{1+a}{(n-1) a}\sum_{k=1}^{n-1} \sum_{j=1}^{n-1}v_j  \frac{z_k^{j}}{(z_k-1)^2}
    \end{equation}
    
    \noindent And again, for $a=3$  we get
    \begin{equation}
        \overline{cr}(S)=\frac{2}{3(n-1)}\sum_{k=1}^{n-1} \sum_{j=1}^{n-1}v_j  \frac{z_k^{j}}{(z_k-1)^2}=\sum_{j=1}^{n-1}v_j \left( \frac{2}{3(n-1)}\sum_{k=1}^{n-1} \frac{z_k^{j}}{(z_k-1)^2} \right),
    \end{equation}
    where the $z_k$ are the $(n-1)$-th roots of $-3$.\\
    
    Finally, using Theorem~\ref{thm:Aziz} for $t=1$ and for the $E_{\leq k}$ polynomial $p_E(z)$, we get for any non-zero complex number $a\neq -1$,
    \begin{equation}
        9{{n}\choose{4}}-\overline{cr}(S) = \frac{3(n-3)}{1+a}{{n}\choose{3}}+ \frac{1+a}{(n-3) a}\sum_{k=1}^{n-3} \sum_{j=0}^{n-3} E_{\leq j} \frac{z_k^{j+1}}{(z_k-1)^2}
    \end{equation}
    
    \noindent and taking now $a=\frac{1}{3}$ we get
    \begin{equation}
        \overline{cr}(S)=-\frac{4}{n-3}\sum_{k=1}^{n-3} \sum_{j=0}^{n-3}E_{\leq j}  \frac{z_k^{j+1}}{(z_k-1)^2}=\sum_{j=0}^{n-3}E_{\leq j} \left( \frac{-4}{n-3}\sum_{k=1}^{n-3} \frac{z_k^{j+1}}{(z_k-1)^2} \right),
    \end{equation}
    where the $z_k$ are now the $(n-3)$-th roots of $-\frac{1}{3}$.
\end{proof}

\section{On the roots of Voronoi, circle and $E_{\leq_k}$ polynomials}\label{sec:rootsOfPolynomials}

In this section, we present properties of the roots of our polynomials.
Note that, by Proposition~\ref{prop:equalpoly}, the Voronoi polynomial $p_V(z)$ has the same roots as the circle polynomial $p_C(z)$, plus the additional root $z=-1$.

A direct relation between roots of polynomials and the rectilinear crossing number can be derived from the well-known relation 
\begin{equation}\label{eq:P'(z)/P(z)}
    \frac{P'(z)}{P(z)}=\sum_{i=1}^{n} \frac{1}{z-a_i},
\end{equation}
where $P(z)$ is a polynomial of degree $n$ with roots $a_1,\ldots, a_n$, and $z$ is any complex number such that $P(z)\neq 0.$\\

When we consider the circle polynomial $p_C(z)$ and $z=1$, using Proposition~\ref{prop:circleeq} we get
\begin{equation}\label{eqn:circleroots}
    \frac{ {{n}\choose{4}} + \overline{cr}(S)  }{ {{n}\choose{3}} } = \sum_{i=1}^{n-3}\frac{1}{1-a_i},
\end{equation}
where the $a_i$ are the roots of $p_C(z)=\sum_{k=0}^{n-3} c_k z^{k}.$ \\

Consider the reciprocal polynomial $p^*_C(z)=\sum_{k=0}^{n-3} c_k z^{n-k-3}$ of $p_C(z)$.
It is well known that the roots of a reciprocal polynomial  $p^*_C(z)$ are $1/a_i$ where $a_i$ is a root of $p_C(z).$
\begin{proposition}
    \begin{equation}\label{eq:circlereverse}
        \sum_{k=0}^{n-3} (n-k-3)c_k = 3{{n}\choose{4}} - \overline{cr}(S)
    \end{equation}
\end{proposition}
\begin{proof}
    We use Equation~(\ref{eq:P'(z)/P(z)}).
    Observe that for a root $a_i$ of $p_C(z)$ and a root $1/a_i$ of $p^*_C(z)$ we have $\frac{1}{1-a_i} + \frac{1}{1-1/a_i}=1.$ Then,
    \[\frac{p'_C(1)}{p_C(1)} + \frac{p^{*'}_C(1)}{p^*_C(1)} = n-3.\]
    Equation~(\ref{eq:circlereverse}) follows by substituting  $p'_C(1)={{n}\choose{4}} +\overline{cr}(S)$, and $p_C(1)=p^*_C(1)={{n}\choose{3}}.$
\end{proof}

The following equation is obtained in a similar fashion:
\begin{proposition}\label{prop:roots}
    \begin{equation}
        \frac{ -2{{n}\choose{4}}+2 \overline{cr}(S)   }{ {{n}\choose{3}}  } = \sum_{i=1}^{n-3}\frac{1+a_i}{1-a_i},
    \end{equation}
    where the $a_i$ are the roots of $p_C(z)=\sum_{k=0}^{n-3}  c_k z^{k}.$\\
\end{proposition}
\begin{proof}
    \[\frac{p'_C(1)}{p_C(1)} - \frac{p^{*'}_C(1)}{p^*_C(1)} =  \sum_{i=1}^{n-3} \frac{1}{1-a_i} - \sum_{i=1}^{n-3} \frac{1}{1-1/a_i} = \sum_{i=1}^{n-3} \frac{1+a_i}{1-a_i}.\]
    The left side of this equation simplifies to $\frac{ -2{{n}\choose{4}}+2 \overline{cr}(S)   }{ {{n}\choose{3}}}  $.
\end{proof}

Note  that Proposition~\ref{prop:roots} also works for the roots of the Voronoi polynomial $p_V(z)$, since the Voronoi polynomial has the same roots as the circle polynomial, plus the additional root $z=-1$ for which the corresponding term in the sum is zero.

Of particular interest is the polynomial $p_V(z)=\sum_{k=1}^{n-1}  v_k z^{k-1}$ for a set of $n$ points in convex position. By Equation (\ref{eqn:numverticespv}), $v_k =(2k - 1)  n - 2  k^2$. By Equation~(\ref{eq:palindromic}), $p_V(z)$ is a palindromic polynomial, so it has roots $a_i$ and $1/a_i$. It follows that for sets $S$ of $n$ points in convex position
\begin{equation}
    \sum_{i=1}^{n-2}\frac{1}{1-a_i}=\frac{n-2}{2},
\end{equation}
where the $a_i$ are the roots of $p_V(z)=\sum_{k=1}^{n-1}  v_k z^{k-1}$.
We also have that for $S$ in convex position, from Proposition~\ref{prop:roots},
\begin{equation}
    \sum_{i=1}^{n-2}\frac{1+a_i}{1-a_i}=0.
\end{equation}

For our next result we use the following theorem due to~\cite{M69}, also see the book by ~\cite{RS02}, Corollary 14.4.2.

\begin{theorem}[Rahman and Schmeisser, 2002]\label{zerosDisk}
        Let $f$ be a polynomial of degree $n$ having all its zeros in the closed disk $|z|\leq k$, where $k\leq 1$. Then
        \[\max_{|z|=1} |f'(z)| \geq  \frac{n}{1+k} \max_{|z|=1} |f(z)|.\]
\end{theorem}

\begin{theorem}\label{thm:convex}
    Let $S$ be a set of $n$ points in general position in the plane. Then $S$ is in convex position if and only if all the roots of the Voronoi polynomial of $S$, $p_V(z)=\sum_{k=1}^{n-1}  v_k z^{k-1}$, lie on the unit circle $\{z:|z|=1\}$.
\end{theorem}
\begin{proof}
    In view of Proposition~\ref{prop:equalpoly}, we can consider the circle polynomial $p_C(z)$ instead of the Voronoi polynomial.
    
    If all the roots of $p_C(z)$  lie on the unit circle, in particular they are contained in the closed unit disk. Thus, we can apply Theorem~\ref{zerosDisk} to the circle polynomial. Since all its coefficients are positive, $\max_{|z|=1} |f'(z)| = f'(1)$ and $\max_{|z|=1} |f(z)| = f(1)$. 
    By Proposition~\ref{prop:circleeq}, we get
    \[{{n}\choose{4}}+\overline{cr}(S) \geq \frac{n-3}{1+k} {{n}\choose{3}}.\] Observing that $(n-3){{n}\choose{3}}=4{{n}\choose{4}}$ and $k\leq 1$, this simplifies to
    \[\overline{cr}(S) \geq {{n}\choose{4}}\frac{3-k}{1+k} \geq {{n}\choose{4}}.\]
    Then, since the inequality $\overline{cr}(S) \leq {{n}\choose{4}}$ always holds, we have $\overline{cr}(S) = {{n}\choose{4}}$, which implies that (is actually equivalent to) $S$ being in convex position.
    
    It remains to show that for a point set $S$ in convex position all its roots lie on the unit circle.
    For $S$ in convex position we have equality in Equations~(\ref{inequ1}) and~(\ref{inequ2}), which implies that $p_C(z)$ is a palindromic polynomial. 
    By Equation~(\ref{eqn:concave}), for $S$ in convex position the coefficients of $p_C(z)$ satisfy 
		$2c_k = c_{k-1} + c_{k+1} + 2$.
    For a palindromic polynomial (with say coefficients $c_k$) that satisfies $2c_k \geq c_{k-1} + c_{k+1}$ it is known that all its roots lie on the unit circle, see e.g.~\cite{CH20}, Theorem 2.5.
\end{proof}

In order to find a lower bound on the largest modulus of the roots of $p_C(z)$ with $S$ not in convex position, we use the following two theorems. The first one is due to~\cite{L1878}, Theorem 1, see also \cite{Nagy1933}, and ~\cite{RS02}, Theorem 3.2.1b.

\begin{theorem}[Rahman and Schmeisser, 2002]\label{thm:sepcircle}
    Let $f$ be a polynomial of degree $n \geq 2$. If $\zeta$, belonging to the complex plane, is neither a zero nor a critical point of $f$, then every circle $C$ that passes through $\zeta$ and $\zeta - n\frac{f(\zeta)}{f'(\zeta)}$ separates at least two zeros of $f$ unless the zeros all lie on $C$.
\end{theorem}

The second theorem is due to~\cite{Ob23}, also see~\cite{BE15} and~\cite{Ma66}, Chapter IX, 41, Exercise 5.
\begin{theorem}[Obrechkoff, 1923]
\label{thm:Obrechkoff}
    For every polynomial $P$ of degree $n$ with non-negative coefficients and every $\theta \in \left(0, \frac{\pi}{2}\right)$, the number of roots in the sector 
    $\{z \in \mathbb{C}\backslash\{0\} \ | \ \ |arg(z) | \leq \theta\}$ is at most $\frac{2\theta n}{\pi}$.
\end{theorem}

\begin{theorem}\label{theorem:VPRootModulus}
    For every set $S$ of $n>3$ points in general position with rectilinear crossing number $\overline{cr}(S)=\alpha\cdot{{n}\choose{4}}$, the Voronoi polynomial $p_V(z)=\sum_{k=1}^{n-1} v_k z^{k-1}$ has a root of modulus at least $1+ \frac{(1-\alpha)\pi^2}{16(n-3)^2} +O\left( \frac{1}{n^4} \right)$. 

\end{theorem}
\begin{proof}
    First observe that if $S$ is in convex position, then $\alpha=1$ and the statement of the theorem holds by Theorem~\ref{thm:convex}. Assume then that $S$ is not in convex position. Therefore, not all points of $S$ lie on a common circle.  
    We apply Theorem~\ref{thm:sepcircle} to the circle polynomial $p_C(z)=\sum_{k=0}^{n-3} c_k z^{k}$ and $\zeta=1.$ Thus, by Proposition~\ref{prop:circleeq}, every circle $C$ passing through point $(1,0)$ and through 
    point $\left(1-\frac{(n-3) {{n}\choose{3}}}{ {{n}\choose{4}}+\overline{cr}(S)},0\right) = \left(1-\frac{4}{1+\alpha},0\right)$ separates at least two zeros of $p_C(z)$. The center of the smallest such circle $C^s$ is
     $(\frac{\alpha-1}{\alpha+1},0)$, and the radius of $C^s$ is $\frac{2}{1+\alpha}.$ Let $D^s$ denote the disk with boundary the circle $C^s$.\\
    The unit disk is contained in $D^s$, touching it in the point $(1,0)$. This already shows that there exists a zero of $p_C(z)$ that lies outside of the unit disk.\\

    From Obrechkoff's inequality, Theorem~\ref{thm:Obrechkoff}, we obtain that the sector $T= \{z \in \mathbb{C}\backslash\{0\} \ | \ \ |arg(z) | < \frac{\pi}{2(n-3)}\}$ is empty of roots of $p_C(z)$.

    By Theorems~\ref{thm:sepcircle} and~\ref{thm:Obrechkoff}, the region $\mathbb{C} \backslash ( D^s \cup T)$ contains a root of $p_C(z)$. Then, the modulus of the largest root of $p_C(z)$ is at least the distance $d$ from the origin to the intersection point in the first quadrant of $C^S$ with the line $y= \tan{\left(\frac{\pi}{2(n-3)}\right)}x$ for $n>4$; see Figure~\ref{lower-bound-pv}.
    
  \begin{figure}[ht!]
  	\centering
  		\includegraphics[scale=0.75]{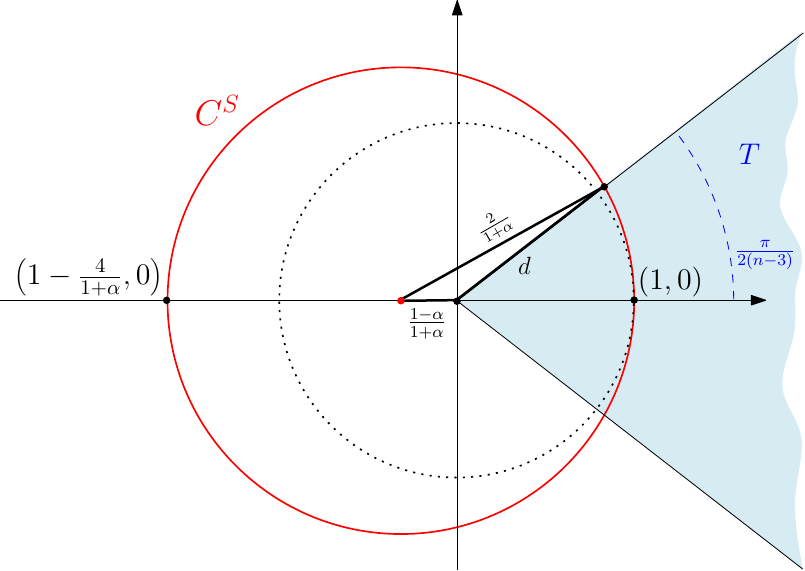}
  	\caption{ The sector $T= \{z \in \mathbb{C}\backslash\{0\} \ | \ \ |arg(z) | < \frac{\pi}{2(n-3)}\}$ is empty of roots of $p_C(z)$. The region $\mathbb{C} \backslash ( D^s \cup T)$ contains a root of $p_C(z)$, with modulus  at least $d$. The unit circle is drawn dotted.
  	}
  	\label{lower-bound-pv}
  \end{figure}
    
    This distance $d$ satisfies, see Figure~\ref{lower-bound-pv}, 
    \[\left(\frac{2}{1+\alpha}\right)^2 = \left(\frac{\alpha-1}{\alpha+1}\right)^2 + d^2 - 2d\frac{1-\alpha}{1+\alpha}\cos{\left(\pi - \frac{\pi}{2(n-3)}\right)}.\]
    Since $\cos{\left(\pi - \frac{\pi}{2(n-3)}\right)}= - \cos{\left(\frac{\pi}{2(n-3)}\right)}$, we get
    \[ d^2 + 2d\frac{1-\alpha}{1+\alpha}\cos{\left( \frac{\pi}{2(n-3)}\right)}+ \frac{\alpha-3}{1+\alpha} =0, \ \mbox{and then}\] 
     \[d=\frac{(-1+\alpha)\cos{\left( \frac{\pi}{2(n-3)}\right)}+\sqrt{ \cos^2{\left( \frac{\pi}{2(n-3)}\right)}(-1+\alpha)^2 -\alpha^2 +2\alpha+3}  }{1+\alpha}.\]
    For $\alpha$ fixed, the Taylor polynomial of \[f(x)=\frac{(-1+\alpha)\cos{x}+\sqrt{ \cos^2{x}(-1+\alpha)^2 -\alpha^2 +2\alpha+3}  }{1+\alpha}\] of degree four at the point $0$ is
    \[1+\frac{1-\alpha}{4}x^2+\frac{-3\alpha^3+5\alpha^2-17\alpha+5}{192}x^4.\]

    Then \[d=f\left( \frac{\pi}{2(n-3)} \right) = 1+ \frac{(1-\alpha)\pi^2}{16(n-3)^2} +O\left( \frac{1}{n^4} \right).\]

\end{proof}

\begin{figure}[ht!]
	\centering
		\includegraphics[scale=0.37]{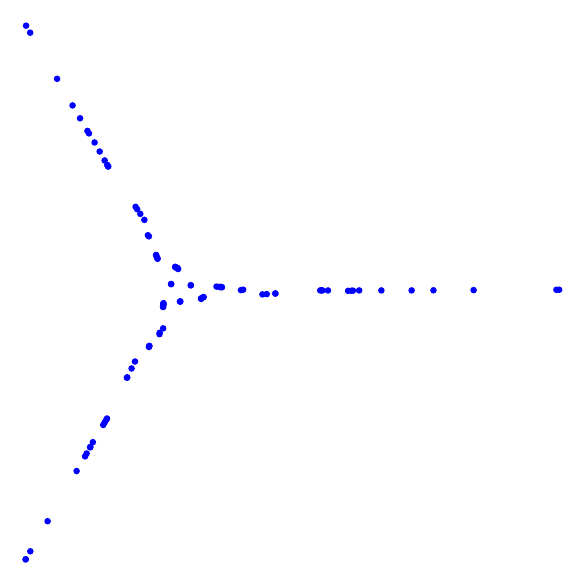}\quad
		\includegraphics[scale=0.37]{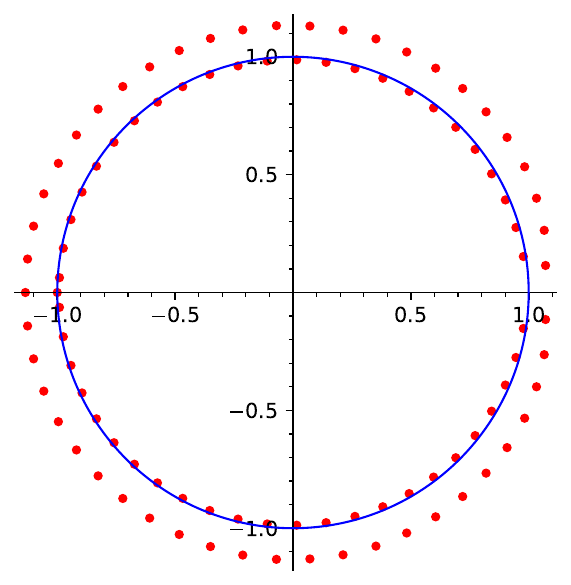}\quad
	\includegraphics[scale=0.37]{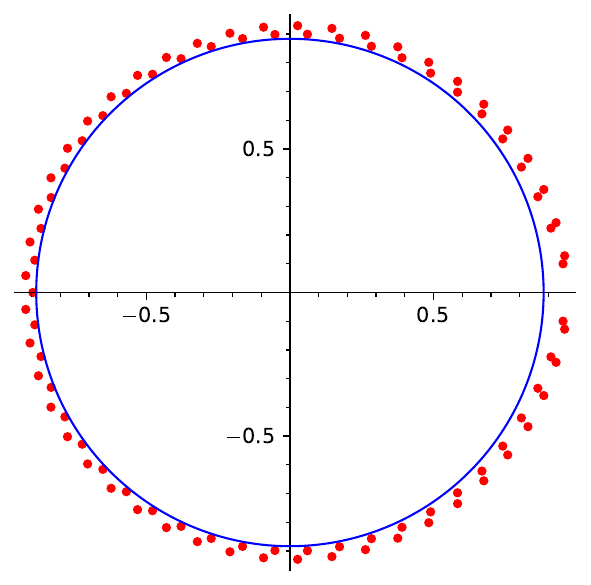}
 \caption{\small Left: Best known example for 
        minimizing $\overline{cr}(S)$ for $|S|=100$, from ~\cite{Aurl}. Roots of its Voronoi (center), and $E_{\leq k}$ (right) polynomials, with circles illustrating, respectively, the bounds of Theorems~\ref{theorem:VPRootModulus} and~\ref{theorem:atmjedgeRootModulusBound}.}
	\label{fig:100}
\end{figure}

We further show that the Voronoi polynomial $p_V(z)$ has a root close to point $1$ in the complex plane. For this, we use Theorem 1.2 from~\cite{MS19}.

\begin{theorem}[Michelen and Sahasrabudhe, 2019]\label{thm:michelen}
Let $X \in \{0,\ldots, n\}$ be a random variable with mean $\mu$, standard deviation $\sigma$ and probability generating function $f_X$ and set $X^* = (X-\mu)\sigma^{-1}.$ If $\delta \in (0,1)$ is such that $|1-\zeta| \geq \delta$ for all roots $\zeta$ of $f_X$ then
\begin{equation}\label{eqn:normal}
 \sup_{t \in \mathbb{R}} | \mathbb{P} (X^* \leq t) - \mathbb{P}(Z \leq t) | =O\left( \frac{\log{n}}{\sigma \delta} \right), 
\end{equation}
where $Z \sim N(0,1).$  
\end{theorem}

\begin{theorem}\label{thm:normal1}
    Let $\alpha$ be a constant from $(0,1]$ and let $S$ be a set of $n$ points in general position in the plane with $\frac{\overline{cr}(S)}{{{n}\choose{4}}}=\alpha$.
    Then, the Voronoi polynomial of $S$, $p_V(z)=\sum_{k=1}^{n-1} v_k z^{k-1}$, has a root $\zeta$ such that $|1-\zeta |  \in  o\left(\frac{log{ (n) }}{n}\right).$
\end{theorem}
\begin{proof}
    Let $X$ be the random variable that counts the number of points of $S$ enclosed by the circle defined by three points chosen uniformly at random from $S$. The probability generating function $f_X$ of $X$ is $p_C(z)/{{n}\choose{3}}$. The roots of $f_X$ are also roots of $p_V(z)$, by Proposition~\ref{prop:equalpoly}.
    From Equations~(\ref{eq:main}) and~(\ref{eq:mainsquare}), see~\cite{circle12}, the mean of $X$ is $\mu= \sum_{k=0}^{n-3} k \cdot \frac{c_k}{{{n}\choose{3}}}=\frac{p'_C(1)}{p_C(1)}=\frac{(1+\alpha)\cdot(n-3)}{4}$ and the standard deviation of $X$ is 
    $\sigma = \sqrt{ \frac{p''_C(1)+p'_C(1)}{p_C(1)}- \left( \frac{p'_C(1)}{p_C(1)}  \right)^2 } =(n-3)\cdot \sqrt{\frac{\alpha}{8} - \frac{\alpha^2}{16}-\frac{1}{80}+\frac{1}{5(n-3)}}.$
    For $\delta \in  o\left( \frac{log{(n)}}{n} \right)$, we have  $ \sigma \delta (\log n)^{-1} \rightarrow \infty$. 
    In order to apply Theorem~\ref{thm:michelen}, it remains to show that $X$ does not approach a normal distribution when $n$ tends towards infinity.
    For $t=0$ in Equation~(\ref{eqn:normal}), we get from Proposition~\ref{prop:circleeq},
    \[P(X^* \leq 0) = \mathbb{P}(X \leq \mu) = \mathbb{P}\left (X \leq \frac{(1+\alpha)(n-3)}{4}\right).\] 
    By Equation~(\ref{inequ1}), $\mathbb{P}(X=k) \geq \frac{ (k+1)(n-k-2)}{ {{n}\choose{3}}}$ holds for $k< \frac{n-3}{2}.$
    Then 
    \[\mathbb{P}(X \leq \mu) \geq \sum_{k=0}^{(1+\alpha)(n-3)/4} \frac{ (k+1)(n-k-2)}{ {{n}\choose{3}}}
    \geq \frac{ (5-\alpha)(1+\alpha)^2}{32}-O\left(\frac{1}{n}\right).\]
    For $\alpha < 1$, 
    \[\lim_{n\rightarrow \infty}  | \mathbb{P} (X^* \leq 0) - \mathbb{P}(Z \leq 0) | = \left| \frac{(5-\alpha)(1+\alpha)^2}{32}-O\left(\frac{1}{n}\right) -\frac{1}{2}\right| >0,\]
     but  $\lim_{n\rightarrow \infty} \frac{\log{n}}{\sigma \delta} =0.$ Then, Equation~(\ref{eqn:normal}) does not hold for $n$ large enough. It follows that there exists a root $\zeta$ of $f_X$ with 
    $|1-\zeta| < \delta$.
    
    For $\alpha=1$, we know $f_X$ precisely. $S$ is a set of $n$ points in convex position and $c_k=(k-1)(n-k-2)$ for every $0 \leq k \leq n-3$. By Equation~(\ref{eqn:concave}), $f_X$ is a concave function. Then $X$ does not approach a normal distribution also for $\alpha=1.$ It follows that there exists a root $\zeta$ of $f_X$ with 
    $|1-\zeta| < \delta$.
\end{proof}

Note that on the other hand, Obrechkoff's Theorem~\ref{thm:Obrechkoff} implies that no root of the Voronoi polynomial can be too close to point 1 in the complex plane. Indeed, the largest disk centered at 1 and contained in the sector $T$ bounded by the two lines $y=\pm \tan \left( \frac {\pi} {2(n-3)} \right)$ in the right half-plane (see Figure \ref{lower-bound-pv}) given by Obrechkoff's Theorem, has radius $r=\sin \left(\frac {\pi} {2(n-3)}\right)$. Then, the distance of the closest root to point 1 is at least  $r\geq \sin\left( \frac {\pi} {2(n-3)} \right)\geq \frac {1} {n-3}$.

In the following, we study the location of the roots of the $E_{\leq k}$ polynomial $p_{E}(z) = \sum_{k=0}^{n-3} E_{\leq k} z^k$ of a point set $S$. Note that its coefficients $E_{\leq k}$ form an increasing sequence of positive numbers. The well-known Eneström-Kakeya theorem,~\cite{kakeya12}, tells us that all the roots of $p_{E}(z)$ are contained in the unit disk, and more precisely, that they are contained in an annulus: The absolute values of the roots of $p_{E}(z)$ lie between the greatest and the least of the $n-3$ quotients
\[\frac{E_{\leq n-4} }{E_{\leq n-3}}, \frac{E_{\leq n-5} }{E_{\leq n-4}}, \ldots, \frac{E_{\leq 1} }{E_{\leq 2}},\frac{E_{\leq 0} }{E_{\leq 1}}.\]

We give a lower bound on the largest modulus of the roots of the $E_{\leq k}$ polynomial.

\begin{theorem}
    Let $S$ be a set of $n>3$ points in general position in the plane, with rectilinear crossing number $\overline{cr}(S)=\alpha\cdot{{n}\choose{4}}$. 
		Then, the $E_{\leq k}$ polynomial of $S$, $p_{E}(z) = \sum_{k=0}^{n-3} E_{\leq k} z^k$, has a root of modulus at least 
		$\frac{3+\alpha}{9-\alpha}.$

\end{theorem}
\begin{proof}
    From Proposition~\ref{prop:edgepoly} we get
    \begin{equation}\label{eqn:edgeroots}
        \frac{p_E'(1)}{p_E(1)}= \frac{ 9{{n}\choose{4}} - \overline{cr}(S)  }{ 3{{n}\choose{3}} } = 
         \frac{ (9-\alpha)\cdot(n-3) }{ 12 }  = \sum_{j=1}^{n-3}\frac{1}{1-a_j},
    \end{equation}
    where the $a_j$, for $j=1,\ldots,n-3,$ are the roots of $p_E(z)$.\\
    
    The conjugate of a complex number $w=a+ib$ is denoted as $\overline{w}=a-ib$.
    The following identity which holds for any complex number $w \neq 1$ is easily verified:
    \begin{equation}\label{eqn:conjugateroots}
        \frac{1}{1-w} + \frac{1}{1-\overline{w}} = 1+ \frac{ 1-|w|^2 }{| 1- w|^2}.
    \end{equation}
    Since all the coefficients of $p_E(z)$ are real numbers, the non-real roots of $p_E(z)$ come in pairs, $a_j$ together with its conjugate $\overline{a_j}.$ We apply Equation~(\ref{eqn:conjugateroots}) to each such pair $w=a_j$ and $\overline{w}=\overline{a_j}.$ For each real root $a_j$ of $p_E(z)$, Equation~(\ref{eqn:conjugateroots}) gives
    \[\frac{1}{1-a_j} = \frac{1}{2}\left(1 + \frac{ 1-|a_j|^2 }{ | 1- a_j|^2}\right).
    \]
    Then, 
    \[\sum_{j=1}^{n-3}\frac{1}{1-a_j} = \sum_{j=1}^{n-3} \frac{1}{2}\left(1 + \frac{ 1-|a_j|^2 }{ | 1- a_j|^2}\right).\]
    We substitute into Equation~(\ref{eqn:edgeroots}) and obtain
    \begin{equation}\label{eqn:useful}
        \frac{(n-3)\cdot(3-\alpha)}{6} = \sum_{j=1}^{n-3} \frac{ 1-|a_j|^2 }{ | 1- a_j|^2}.
    \end{equation}
    The next inequality, which holds for every complex number $w\neq 1$ with modulus $|w| < 1$ is easily verified.
    \begin{equation}\label{ineqn:modulusbound}
        \frac{1-|w|^2}{|1-w|^2} \geq \frac{1-|w|}{1+|w|} 
    \end{equation}
    Since the coefficients of $p_{E}(z)$ form an increasing sequence of positive numbers, the roots $a_j$ of $p_E(z)$ all satisfy $|a_j|<1.$ We use Inequality~(\ref{ineqn:modulusbound}) in Equation~(\ref{eqn:useful}).  Then,
    \[
    \frac{(n-3)\cdot(3-\alpha)}{6} \geq (n-3)\cdot \min_{ j=1,\ldots, n-3} \frac{1-|a_j|}{1+|a_j|}. 
    \]
    Note that the function $f(x)=\frac{1-x}{1+x}$ is decreasing for $x \in (0,1).$ Then, the minimum of $\frac{1-|a_j|}{1+|a_j|}$ is attained when $|a_j|$ is maximum.
		It follows that $ |a_{max}| = \max\limits_{j=1,\ldots, n-3}|a_j| \geq \frac{3+\alpha}{9-\alpha}$, 
		by  solving for $|a_{max}|$ in 
		\begin{equation}\label{eq:almost}
		    \frac{(n-3)\cdot(3-\alpha)}{6} \geq (n-3)\cdot  \frac{1-|a_{max}|}{1+|a_{max}|}.
		\end{equation}
This completes the proof. 
\end{proof}

Next, we show a better lower bound on the largest modulus of $p_E(z)$ when $n$ is large enough.
For this, we use the following theorem of \cite{T39} (page 171), also see \cite{GS13}, Theorem A.

\begin{theorem}[Gardner and Shields, 2013]\label{thm:Titchmarsh}
    Let $F(z)$ be analytic in $|z| \leq R.$ Let $|F(z)| \leq M$ in the disk $|z| \leq R$ and suppose $F(0) \neq 0$. Then, for $0 < \delta <1$, the number of zeros of $F(z)$ in the disk $|z| \leq \delta R$ is less than
    \[\frac{1}{\log{\frac{1}{\delta}}} \log{\frac{M}{|F(0)|}}.
    \]
\end{theorem}

\begin{theorem} \label{theorem:atmjedgeRootModulusBound}
    Let $S$ be a set of $n$ points with $h$ of them on the boundary of the convex hull of $S$, in general position in the plane. Then $p_{E}(z) = \sum_{k=0}^{n-3} E_{\leq k} z^k$ has a root of modulus at least
        \[\left( \frac{3  {{n}\choose{3}}} {h}\right)^{-\frac{1}{n-3}}.\]
\end{theorem}
\begin{proof}
    By the maximum modulus principle and Proposition~\ref{prop:edgepoly}, $|p_E(z)| \leq 3{{n}\choose{3}}$ in the disk $|z| \leq 1$, and $p_E(0) = E_{\leq 0} = h \neq 0.$
    We apply Theorem~\ref{thm:Titchmarsh} with a large $\delta$ from the interval $(0,1)$ such that \[\frac{1}{\log{\frac{1}{\delta}}} \log{\frac{3{{n}\choose{3}}  }{h}} < n-3.
    \]
    Then, $p_E(z)$ has at least one root of modulus greater than $\delta.$
    Solving for $\delta$, we obtain that $\delta < \left( \frac{3{{n}\choose{3}}}{h}\right)^{-\frac{1}{n-3}}$.
\end{proof}

For an illustration of Theorem~\ref{theorem:atmjedgeRootModulusBound}, see Figure~\ref{fig:100} (right) on page \pageref{fig:100}.

\section{Discussion}\label{sec:discussion}

We have introduced three polynomials $p_V(z)$, $p_C(z)$, $p_{E}(z)$ for sets $S$ of $n$ points in general position in the plane, showing their connection to $\overline{cr}(S)$ and several bounds on the location of their roots.
The obvious open problem is using bounds on such roots to improve upon the current best bound on the rectilinear crossing number problem. 
Besides, we think that the presented polynomials are interesting objects of study on their own, also given the many applications of Voronoi diagrams.

For some of the formulas presented for one of the polynomials, such as for example Equations~(\ref{eqn:circleroots}) and~(\ref{eqn:useful}), there are analogous statements for the other considered polynomials. These can be derived easily and are omitted.

Further, several other polynomials on point sets can be considered. The reader interested in crossing numbers probably has in mind the $j$-edge polynomial $p_e(z)=\sum_{j=0}^{n-2} e_j z^j$ of a point set $S$. The known formula for the rectilinear crossing number $\overline{cr}(S)$ in terms of the numbers of $j$-edges $e_j$ of $S$, 
see~\cite{LVW04}, Lemma 5, translates into
\[2\overline{cr}(S) -6{{n}\choose{4}} = p''_e(1) - (n-3)p'_e(1).\]  
As is the case for the Voronoi polynomial and the circle polynomial, for sets $S$ of $n$ points in convex position, the $j$-edge polynomial $p_e(z)$ has all its roots on the unit circle. This is readily seen since $p_e(z)$ is then  $n$ times the all ones polynomial, $p_e(z)= n\sum_{j=0}^{n-2} z^j$, as $e_j=n$ for all $j$, if $S$ is in convex position. Its roots are the $(n-1$)th roots of unity, except $z=1.$

For point sets that minimize $\overline{cr(S)}$, the roots of the Voronoi polynomial seem to be close to two circles, such as in the example of Figure~\ref{fig:100}. For arbitrary point sets this phenomenon does not occur, but for $n$ large enough, the roots tend to lie close to the unit circle, also see Theorem 1 in \cite{hughes2008zeros}.

We finally propose to study the presented polynomials for random point sets.  The expected rectilinear crossing number is known for sets of $n$ points chosen uniformly at random from a convex set $K$, for several shapes of $K$, see e.g.~\cite{S04}, Section 1.4.5. pp. 63-64, and the survey of~\cite{AFS12}.

\acknowledgements
\label{sec:ack}
M. C., A. d.-P., C. H., and D. O. were supported by project PID2019-104129GB-I00/MICIU/AEI/10.13039 /501100011033. 
C.H. was supported by project Gen. Cat. DGR 2021-SGR-00266.
D. F.-P. was supported by grants PAPIIT IN120520 and PAPIIT IN115923 (UNAM, México).

\nocite{*}
\bibliographystyle{abbrvnat}
\bibliography{dmtcs}
\label{sec:biblio}

\end{document}